\documentclass[11pt]{article}
\usepackage[margin=1in]{geometry}
\usepackage{amsfonts}
\usepackage{amsmath}
\usepackage{amssymb}
\usepackage{amsthm}
\usepackage{array}
\usepackage[UKenglish]{babel}
\usepackage{color}
\usepackage{enumerate}
\usepackage{graphicx}
\usepackage{hyperref}
\usepackage{latexsym}
\usepackage{multirow}
\usepackage{pdfpages}
\usepackage{pgfplots}
\usepackage{tikz}
\usepackage{tkz-fct}
\usepackage{tkz-graph}
\usepackage{upgreek}
\usepackage{genyoungtabtikz}
\usetikzlibrary{shapes}
\usetikzlibrary{arrows}

\usepackage{standalone}
\usepackage{slantsc}
\usepackage{lmodern}
\theoremstyle{plain}
\newtheorem{theorem}{Theorem}
\newtheorem{cor}[theorem]{Corollary}
\newtheorem{lemma}[theorem]{Lemma}

\newtheorem{conj}[theorem]{Conjecture}
\theoremstyle{definition}

\theoremstyle{remark}

\usepackage{caption}
\usepackage{subcaption}
\DeclareCaptionType{algorithm}
\newcommand{\bk}{\backslash}
\newcommand{\lm}{\lambda}

\begin{document}
\title{A Deletion-Contraction Relation for the Chromatic Symmetric Function}
\author{Logan Crew\thanks{Department of Mathematics, University of Pennsylvania, Philadelphia, PA. Email: {\tt crewl@math.upenn.edu}}\,\, and Sophie Spirkl\thanks{Department of Mathematics, Princeton University, Princeton, NJ. Email: {\tt sspirkl@math.princeton.edu}.}}
\date{\today}

\maketitle

\begin{abstract}
  We extend the definition of the chromatic symmetric function $X_G$ to include graphs $G$ with a vertex-weight function $w : V(G) \rightarrow \mathbb{N}$.  We show how this provides the chromatic symmetric function with a natural deletion-contraction relation analogous to that of the chromatic polynomial.  Using this relation we derive new properties of the chromatic symmetric function, and we give alternate proofs of many fundamental properties of $X_G$.
\end{abstract}


\begin{section}{Introduction}\label{Introduction}

  The chromatic symmetric function $X_G$ of a graph $G$ was introduced by Stanley in 1995 \cite{stanley} as a generalization of the chromatic polynomial $\chi_G(x)$.  It is defined as
  $$
X_G(x_1,x_2,\dots) = \sum_{\kappa} \prod_{v \in V(G)} x_{\kappa(v)}
  $$
where the sum ranges over all proper colorings $\kappa$ of $G$.  Recent research on $X_G$ has focused on (among other topics) the Stanley-Stembridge conjecture that the chromatic symmetric function of the incomparability graph of a $(3+1)$-free poset is $e$-positive \cite{huh, epos, dahl, foley, guay, stanley}, the related conjecture that the chromatic symmetric function of a claw-free graph is $s$-positive \cite{gash, pau, paw, wang}, and the conjecture that $X_G$ distinguishes nonisomorphic trees \cite{trees, heil}.  Other results have extended the definition of $X_G$ to include quasisymmetric functions \cite{per, ell, wachs} or noncommuting variables \cite{dahl2, noncomm}.

In this paper we extend $X_G$ in a different direction.  Our motivation is the observation that while the chromatic polynomial $\chi_G$ of a graph $G$ admits the simple edge deletion-contraction relation
$$\chi_G = \chi_{G \bk e}-\chi_{G / e}
$$
the chromatic symmetric function $X_G$ does not admit such a relation.  This is because $X_G$ is always homogeneous of degree $|V(G)|$, so trying to formulate such a relation encounters difficulties when considering edge contraction.  Instead, the best known edge recurrence relation on $X_G$ itself is a triangular relation discovered by Orellana and Scott in 2014 \cite{ore}, and its generalization to all cycles in 2018 by Dahlberg and van Willigenburg \cite{dahl}. 

To provide $X_G$ with a deletion-contraction relation in a natural way, we extend the definition to include pairs $(G,w)$ consisting of a graph $G$ and a vertex-weight function $w : V(G) \rightarrow \mathbb{N}$.  Then the extended function

$$
X_{(G,w)}(x_1,x_2,\dots) = \sum_{\kappa} \prod_{v \in V(G)} x_{\kappa(v)}^{w(v)}
$$
admits a generalization of the classic deletion-contraction relation of the chromatic polynomial of the form

$$
X_{(G,w)} = X_{(G \bk e, w)} - X_{(G / e, w / e)}
$$
where $w / e$ indicates that when we contract an edge $e$, the weight of the contacted vertex is the sum of the weights of the endpoints of $e$.

This approach builds upon previous extensions of the chromatic symmetric function that admit deletion-contraction relations.  An early example was the chromatic symmetric function in noncommuting variables $Y_G(x_1,x_2,\dots)$, introduced by Gebhard and Sagan in 1999 \cite{noncomm}, which satisfies $Y_G = Y_{G \bk e} - Y_{G / e}\uparrow$, where $\uparrow$ is an operation they define called induction.  This induction operation takes a variable $x_i$ representing the color of a vertex and duplicates it, which is similar to our contraction term $X_{(G / e, w / e)}$ that adds vertex weights.  More recently, the pointed chromatic symmetric function $X_{G,v}(t,x_1,x_2,\dots)$, rooted at a vertex $v$, was introduced by Pawlowski in 2018 \cite{paw} and satisfies $X_{G,v} = X_{G \bk e, v} - tX_{G/e,v}$ when $e$ is an edge incident to $v$.

Additionally, $X_{(G,w)}$ relates to work of Noble and Welsh \cite{noble}.  Their $W$-polynomial, defined on vertex-weighted graphs, admits a deletion-contraction relation.  They show that $X_G$ can be recovered from the $W$-polynomial on unweighted graphs, and it can be proved in an analogous way that $X_{(G,w)}$ can be recovered from the $W$-polynomial for vertex-weighted graphs.  Moreover, they demonstrate that in unweighted graphs, the $W$-polynomial contains the same information as $XB_G$, the bad-coloring chromatic symmetric function of Stanley \cite{stanley2}.  A simple modification of their argument shows that the natural extension of $X_{(G,w)}$ to a vertex-weighted form of $XB_G$ contains the same information as the $W$-polynomial for all vertex-weighted graphs.  The authors' paper \cite{paper} provides further research in this direction.

In this paper, we prove new properties of and rederive known results for $X_G$ by proving them for the more general $X_{(G,w)}$, and in the case of previously known results, these proofs are substantially different in nature from the original ones, as they depend primarily on simple enumerative techniques and induction using deletion-contraction.
  
In Section 2 we provide background on graphs and symmetric functions that will be used throughout this paper.  In Section 3 we define formally our extended function on vertex-weighted graphs, prove that a deletion-contraction relation holds, and use it to provide alternate proofs of known results by extending them to the vertex-weighted case.  In Section 4 we derive a new result on the $e$-basis expansion of $X_{(G,w)}$ that generalizes a result of \cite{stanley}.  Finally, in Section 5 we consider further applications of $X_{(G,w)}$.  In particular, we define an extension of $X_{(G,w)}$ that generalizes the chromatic quasisymmetric function of Shareshian and Wachs \cite{wachs} and show that it satisfies a deletion-contraction relation, and we show that $X_{(G,w)}$ is neither $e$-positive nor $e$-negative when $w$ is not the weight function that gives every vertex weight $1$.  We also discuss potential applications of $X_{(G,w)}$ to $s$-positivity, partition systems, the umbral chromatic polynomial, the path-cycle symmetric function for digraphs, and functions on double posets.

\end{section}

\begin{section}{Background}\label{Background}

  An \emph{integer partition} (or just \emph{partition}) is a tuple $\lm = (\lm_1,\dots,\lm_k)$ of positive integers such that $\lm_1 \geq \dots \geq \lm_k$.  The integers $\lm_i$ are the \emph{parts} of $\lm$.  If $\sum_{i=1}^k \lm_i = n$, we say that $\lm$ is a partition of $n$, and we write $\lm \vdash n$, or $|\lm| = n$.  The number of parts $k$ is the \emph{length} of $\lm$, and is denoted by $l(\lm)$.  The number of parts equal to $i$ in $\lm$ is given by $r_i(\lm)$.

  The \emph{Young diagram} $Y(\lm)$ of shape $\lm$ is a set of boxes, left- and top-justified, so that there are $\lm_i$ boxes in the $i^{th}$ row from the top.  A \emph{filling} of $Y(\lm)$ with elements from a set $S$ is an assignment of an $s \in S$ to each box of $Y(\lm)$.  A \emph{Young tableau} $T$ of shape $\lm$ is a filling of $Y(\lm)$ with positive integers $\mathbb{Z}^+$.  A Young tableau $T$ is \emph{semi-standard} if
  \begin{itemize}
    \item Its integers are \emph{weakly increasing} along rows left-to-right.
    \item Its integers are \emph{strictly increasing} down columns.
  \end{itemize}
  For example, the following is a semi-standard Young tableau of shape $(4,2,2,1)$:
  $$\young(1124,23,44,6)$$

  A function $f(x_1,x_2,\dots) \in \mathbb{R}[[x_1,x_2,\dots]]$ is \emph{symmetric} if $f(x_1,x_2,\dots) = f(x_{\sigma(1)},x_{\sigma(2)},\dots)$ for every permutation $\sigma$ of the positive integers $\mathbb{N}$.  The \emph{algebra of symmetric functions} $\Lambda$ is the subalgebra of $\mathbb{R}[[x_1,x_2,\dots]]$ consisting of those symmetric functions $f$ that are of bounded degree (that is, there exists a positive integer $n$ such that every monomial of $f$ has degree $\leq n$).  Furthermore, $\Lambda$ is a graded algebra, with natural grading
  $$
  \Lambda = \bigoplus_{k=0}^{\infty} \Lambda^d
  $$
  where $\Lambda^d$ consists of symmetric functions that are homogeneous of degree $d$ \cite{mac,stanleybook}.

  Each $\Lambda^d$ is a finite-dimensional vector space over $\mathbb{R}$, with dimension equal to the number of partitions of $d$ (and thus, $\Lambda$ is an infinite-dimensional vector space over $\mathbb{R}$).  There are five commonly-used bases of $\Lambda$ that are indexed by partitions $\lm = (\lm_1,\dots,\lm_k)$ (for more details see e.g. \cite{mac,stanleybook}):

  The monomial symmetric functions, defined as
  $$
  m_{\lm} = \sum_{\sigma \in S_{\mathbb{N}}} x_{\sigma(1)}^{\lm_1}x_{\sigma(2)}^{\lm_2} \dots x_{\sigma(k)}^{\lm_k}.
  $$

  The power-sum symmetric functions, defined by the equations
  $$
  p_n = \sum_{k=1}^{\infty} x_k^n, \hspace{0.3cm} p_{\lm} = p_{\lm_1}p_{\lm_2} \dots p_{\lm_k}.
  $$

  The elementary symmetric functions, defined by the equations
  $$
  e_n = \sum_{i_1 < \dots < i_n} x_{i_1} \dots x_{i_n}, \hspace{0.3cm} e_{\lm} = e_{\lm_1}e_{\lm_2} \dots e_{\lm_k}.
  $$

  The homogeneous symmetric functions, defined by the equations
  $$
  h_n = \sum_{\mu \vdash n} m_{\mu}, \hspace{0.3cm} h_{\lm} = h_{\lm_1}h_{\lm_2} \dots h_{\lm_k}.
  $$

  And the Schur functions, which may be defined as
  $$
  s_{\lm} = \sum_{\mu \vdash |\lm|} K_{\lm\mu}m_{\mu}
  $$
  where $K_{\lm\mu}$ is the number of semi-standard Young tableaux $T$ of shape $\lm$ where for all positive integers $i$, $T$ contains $\mu_i$ instances of $i$.  Among these bases, in discussing the chromatic symmetric function we will focus mostly on the basis of elementary symmetric functions, with reference also to the power-sum basis (in Section $3$) and the Schur function basis (in Section $5$).

  Given a symmetric function $f$ and a basis $b$ of $\Lambda$, we say that $f$ is \emph{$b$-positive} if when we write $f$ in the basis $b$, all coefficients are nonnegative.

  We define the \emph{symmetric function involution} $\omega$ by $\omega(p_{\lm}) = $ $(-1)^{|\lm|-l(\lm)}p_{\lm}$.

  A \emph{graph} $G = (V,E)$ consists of a \emph{vertex set} $V$ and an \emph{edge multiset} $E$ where the elements of $E$ are pairs of (not necessarily distinct) elements of $V$.  An edge $e \in E$ that contains the same vertex twice is called a \emph{loop}.  If there are two or more edges that each contain the same two vertices, they are called \emph{multi-edges}.  A \emph{simple graph} is a graph $G = (V,E)$ in which $E$ does not contain loops or multi-edges (thus, $E \subseteq \binom{V}{2}$).  If $\{v_1,v_2\}$ is an edge (or nonedge), we will write it as $v_1v_2 = v_2v_1$.  The vertices $v_1$ and $v_2$ are the \emph{endpoints} of the edge $v_1v_2$.  We will use $V(G)$ and $E(G)$ to denote the vertex set and edge multiset of a graph $G$ respectively.

  The \emph{complement} of a simple graph $G = (V,E)$ is denoted $\overline{G}$, and is defined as $\overline{G} = (V, \binom{V}{2} \bk E)$, so in $\overline{G}$ every edge of $G$ is replaced by a nonedge, and every nonedge is replaced by an edge.

  A \emph{subgraph} of a graph $G$ is a graph $G' = (V',E')$ where $V' \subseteq V$ and $E' \subseteq E|_{V'}$, where $E|_{V'}$ is the set of edges with both endpoints in $V'$.  An \emph{induced subgraph} of $G$ is a graph $G' = (V',E|_{V'})$ with $V' \subseteq V$.  A \emph{stable set} of $G$ is a subset $V' \subseteq V$ such that $E|_{V'} = \emptyset$.  A \emph{clique} of $G$ is a subset $V' \subseteq V$ such that for every pair of distinct vertices $v_1$ and $v_2$ of $V'$, $v_1v_2 \in E(G)$.

  A \emph{path} in a graph $G$ is a sequence of edges $v_1v_2$, $v_2v_3$, \dots, $v_{k-1}v_k$ such that $v_i \neq v_j$ for all $i \neq j$.  The vertices $v_1$ and $v_k$ are the \emph{endpoints} of the path.  A \emph{cycle} in a graph is a sequence of edges $v_1v_2$, $v_2v_3$, \dots, $v_kv_1$ such that $v_i \neq v_j$ for all $i \neq j$.  Note that in a simple graph every cycle must have at least $3$ edges, although in a nonsimple graph there may be cycles of size $1$ (a loop) or $2$ (multi-edges).  

  A graph $G$ is \emph{connected} if for every pair of distinct vertices $v_1$ and $v_2$ of $G$ there is a path in $G$ with $v_1$ and $v_2$ as its endpoints.  The \emph{connected components} of $G$ are the maximal induced subgraphs of $G$ which are connected.
  
  The \emph{complete graph} $K_n$ on $n$ vertices is the unique simple graph having all possible edges, that is, $E(K_n) = \binom{V}{2}$.

  Given a graph $G$, there are two commonly used operations that produce new graphs.  One is \emph{deletion}: given an edge $e \in E(G)$, the graph of $G$ \emph{with} $e$ \emph{deleted} is the graph $G' = (V(G), E(G) \bk \{e\})$, and is denoted $G \bk e$.  Likewise, if $S$ is a multiset of edges, we use $G \bk S$ to denote the graph $(V(G),E(G) \bk S)$.

  The other operation is the \emph{contraction of an edge} $e = v_1v_2$, denoted $G / e$.  If $v_1 = v_2$ ($e$ is a loop), we define $G / e = G \bk e$.  Otherwise, we create a new vertex $v^*$, and define $G / e$ as the graph $G'$ with $V(G') = (V(G) \bk \{v_1,v_2\}) \cup v^*$, and $E(G') = (E(G) \bk E(v_1, v_2)) \cup E(v^*)$, where $E(v_1,v_2)$ is the set of edges with at least one of $v_1$ or $v_2$ as an endpoint, and $E(v^*)$ consists of each edge in $E(v_1,v_2) \bk e$ with the endpoint $v_1$ and/or $v_2$ replaced with the new vertex $v^*$.  Note that this is an operation on a (possibly nonsimple) graph that identifies two vertices while keeping and/or creating multi-edges and loops.

  There is also a different version of edge contraction that is defined only on simple graphs.  In the case that $G$ is a simple graph, we define the \emph{simple contraction} $G \nmid e$ to be the same as $G / e$ except that after performing the contraction operation, we delete any loops and all but a single copy of each multi-edge so that the result is again a simple graph.  

  Let $G = (V(G),E(G))$ be a (not necessarily simple) graph.  A map $\kappa: V(G) \rightarrow \mathbb{N}$ is called a \emph{coloring} of $G$.  This coloring is called \emph{proper} if $\kappa(v_1) \neq \kappa(v_2)$ for all $v_1,v_2$ such that there exists an edge $e = v_1v_2$ in $G$.  As was described in Section 1, the \emph{chromatic symmetric function} $X_G$ of $G$ is defined as
  
  $$X_G(x_1,x_2,\dots) = \sum_{\kappa} \prod_{v \in V(G)} x_{\kappa(v)}$$ 
  where the sum runs over all proper colorings $\kappa$ of $G$.  Note that if $G$ contains a loop then $X_G = 0$, and $X_G$ is unchanged by replacing any multi-edges by a single edge.
  
  A thorough overview of $X_G$ is given in \cite{stanley}; we postpone introducing its properties to the next section, where we will also prove generalizations of them.

\end{section}

\begin{section}{Extending $X_G$ to Vertex-Weighted Graphs}\label{Weight}

  Define a \emph{vertex-weighted graph} $(G,w)$ to be a graph $G$ together with a vertex-weight function $w: V(G) \rightarrow \mathbb{N}$.  The \emph{weight of a vertex} $v \in V(G)$ is $w(v)$.  Using the notation
$$
x_{\kappa}(G, w) = \prod_{v \in V(G)} x_{\kappa(v)}^{w(v)}
$$
we generalize the chromatic symmetric function to vertex-weighted graphs as
  \begin{equation}\label{eq:XGW}
    X_{(G,w)} =  \sum_{\kappa} x_{\kappa}(G,w)
  \end{equation}
where the sum is taken over all proper colorings $\kappa$ of $G$.  We use this nonstandard notation as it will be convenient to refer explicitly to individual summands of $X_{(G,w)}$ in proofs.  Note that the usual chromatic symmetric function $X_G$ is equivalent to $X_{(G,w)}$ where $w$ is the function assigning weight $1$ to each vertex.

   Given a vertex-weighted graph $(G,w)$ and $A \subseteq V(G)$, define the \emph{total weight of} $A$, denoted $w(A)$, to be $\sum_{v \in A} w(v)$.  Define the \emph{total weight of} $G$ to be the total weight of $V(G)$.  Throughout this paper, when $G$ is clear we will generally use $n$ to denote the number of vertices of $G$, and $d$ to denote the total weight of $G$.

   Let $\lm = (\lm_1,\dots,\lm_k)$ be a partition.  Define $St_{\lm}(G,w)$ to be the set of (unordered) partitions of $V(G)$ into $k = l(\lm)$ stable sets whose total weights are $\lm_1,\dots,\lm_k$.  We begin by establishing a simple formula for expanding $X_{(G,w)}$ in the monomial basis:

   \begin{lemma}

     If $(G,w)$ is a vertex-weighted graph with $n$ vertices and total weight $d$, then 
    \begin{equation}\label{eq:mbasis}
      X_{(G,w)} = \sum_{\lm \vdash d} |St_{\lm}(G,w)|\left(\prod_{i=1}^{d} r_i(\lm)!\right) m_{\lm}
    \end{equation}
    where we recall that $r_i(\lm)$ is the number of parts of $\lm$ equal to $i$.
  \end{lemma}

   \begin{proof}

     The proof is a simple modification of the proof of (\cite{stanley}, Theorem 2.4).  Since $X_{(G,w)}$ is symmetric, it suffices to show that the coefficient of $x_1^{\lm_1} \dots x_k^{\lm_k}$ is correct.  For every element of $St_{\lm}(G,w)$, label the stable sets $L_1,\dots,L_k$ in some order such that $|L_i| = \lm_i$.  Then there are $\left(\prod_{i=1}^{d} r_i(\lm)!\right)$ corresponding proper colorings $\kappa$ of $(G,w)$ such that $\forall i \exists j$ with $\kappa^{-1}(j) = L_i$ and also $x_{\kappa}(G,w) = x_1^{\lm_1} \dots x_k^{\lm_k}$, since one such coloring is $\kappa(L_i) = i$ for $1 \leq i \leq k$, and we may also permute the colors among those $L_i$ that have the same cardinality.  Since also clearly every proper $\kappa$ with $x_{\kappa}(G,w) = x_1^{\lm_1} \dots x_k^{\lm_k}$ has a corresponding element of $St_{\lm}(G,w)$ for which $\kappa$ is monochromatic on each part, the terms of $X_{(G,w)}$ are in one-to-one correspondence with those of the right-hand side of \eqref{eq:mbasis}, so the lemma is proved. 
     
  \end{proof}

  As an example, for a partition $\lm$, we define $K^{\lm} = (K_{l(\lm)},w)$ where $w(v_i) = \lm_i$ for some ordering $v_1,v_2,\dots,v_n$ of the vertices.  Since the only stable sets of $K_{l(\lm)}$ are single vertices, every coloring of $K_{l(\lm)}$ colors every vertex with a distinct color, and so only monomials of $m_{\lm}$ appear.  Each monomial of $m_{\lm}$ will occur once for each permutation of the colors of vertices with the same weights, and so we have
  $$
  X_{K^{\lm}} = \left(\prod_{i=1}^{\infty} r_i(\lm)!\right) m_{\lm}.
  $$

  Analogously, we define $\overline{K^{\lm}} = (\overline{K_{l(\lm)}}, w)$ where $w(v_i) = \lm_i$ for some ordering of the vertices.  Note that the chromatic symmetric function is multiplicative across disjoint unions, since we may color each of the connected components independently.  Since $X_{K^n} = p_n$, it follows that
  $$
  X_{\overline{K^{\lm}}} = p_{\lm}.
  $$

  Note that in the case of unweighted graphs, there is no $G$ such that $X_G$ is equal to a nonzero multiple of $m_{\lm}$ or $p_{\lm}$ except in the case that $\lm = 1^n$ \cite{cho}.\footnote{It is natural to ask which bases are ``representable'' as chromatic symmetric functions in this way.  In addition to the $m$- and $p$-bases, we obtain $X_G = \left(\prod \lm_i!\right)e_{\lm}$ even in the unweighted case by taking $G$ to be a disjoint union of cliques of sizes equal to the parts of $\lm$.  Furthermore, it is easy to show by combining the $p$-positivity of \eqref{eq:ppositive} with the results of \cite{cho} that there are no vertex-weighted graphs $(G,w)$ such that $X_{(G,w)}$ is a nonzero multiple of $h_{\lm}$ or $s_{\lm}$ except in the case that $\lm = 1^n$.}

\begin{subsection}{A Deletion-Contraction Relation}
  
  One of the primary motivations for extending the chromatic symmetric function to vertex-weighted graphs is the existence of a deletion-contraction relation in this setting.  Given a vertex-weighted graph $(G,w)$, and an edge $e = v_1v_2$ of $G$, let $w/e$ be the modified weight function on $G / e$ such that $w/e = w$ if $e$ is a loop, and otherwise $(w/e)(v) = w(v)$ if $v \neq v_1,v_2$, and for the vertex $v^*$ of $G/e$ formed by the contraction, $(w/e)(v^*) = w(v_1) + w(v_2)$.  Note that the same definition of $w / e$ may be applied to the simple contraction $G \nmid e$, so we use the same notation.

  \begin{lemma}\label{lem:delcon}
    Let $(G,w)$ be a vertex-weighted graph, and let $e \in E(G)$ be any edge.  Then
    \begin{equation}\label{eq:delcon}
      X_{(G,w)} = X_{(G \bk e, w)} - X_{(G / e, w / e)}
    \end{equation}
    and if $G$ is a simple graph,
    $$
    X_{(G,w)} = X_{(G \bk e, w)} - X_{(G \nmid e, w / e)}.
    $$
 
  \end{lemma}

  \begin{proof}

    First, we note that for a simple graph $G$, $X_{(G \nmid e, w / e)} = X_{(G / e, w / e)}$.  This is because the only case in which $G \nmid e$ is different from $G / e$ is in the case that some vertex $v'$ had edges to both endpoints of $e$, for then $G / e$ would have a multi-edge where $G \nmid e$ has a single edge.  But by Lemma 1 multi-edges may be reduced to a single edge without affecting the chromatic symmetric function, establishing the claim.  Thus, it suffices to prove \eqref{eq:delcon}.  

    We rewrite \eqref{eq:delcon} in the form
    \begin{equation}\label{eq:edgeadd}
    X_{(G \bk e, w)} = X_{(G,w)} + X_{(G / e, w / e)}.
    \end{equation}

    The statement is immediate if $e$ is a loop, so we may assume $e = v_1v_2$ connects distinct vertices, and we also let $v^*$ be the contracted vertex in $G / e$.  It suffices to show a one-to-one correspondence between terms of $X_{(G \bk e, w)}$ and terms of $X_{(G,w)}$ or $X_{(G / e, w / e)}$.  We consider two cases for each term $x_{\kappa}(G \bk e, w)$ (as defined in Section 2) occurring in the left-hand side of \eqref{eq:edgeadd} based on the proper coloring $\kappa$.  If $\kappa(v_1) = \kappa(v_2)$, then $x_{\kappa}(G \bk e, w) = x_{\kappa_e}(G / e, w / e)$, where $\kappa_e$ is the proper coloring of $G / e$ such that $\kappa_e(v^*) = \kappa(v_1)$, and for all other vertices $v$, $\kappa_e(v) = \kappa(v)$.  If $\kappa(v_1) \neq \kappa(v_2)$, then $x_{\kappa}(G \bk e, w) = x_{\kappa}(G, w)$.  This correspondence is injective, since changing the color of any vertex in $G \bk e$ changes the corresponding proper coloring of either $G$ or $G / e$.  This correspondence is also surjective, since given a proper coloring of $G$ or $G / e$, we can recover a proper coloring of $G \bk e$ that is its preimage under this map by removing $e$ or uncontracting $v^*$, respectively.

  \end{proof}

  Note that it is also possible to write this relation in the ``vertex uncontraction'' form
  \begin{equation}\label{eq:uncon}
X_{(G / e, w / e)} = X_{(G \bk e, w)} - X_{(G, w)}.
  \end{equation}
  This form has increased flexibility, because if we are given $(G / e, w / e)$, we may make two choices in uncontracting: first, if the vertex being uncontracted has weight greater than $2$, we may choose how to distribute the weights to the two new vertices in $G$, and second, for edges that were incident to the contracted vertex, we may choose how those edges are incident to the newly created vertices in $G$.  Thus, whenever this uncontraction form is used on a graph $(G / e, w / e)$ throughout this paper, we will specify the graph $(G,w)$.  

  One advantage of having a deletion-contraction relation is that to prove a property on graphs, we can pass to an appropriate property on vertex-weighted graphs, and either use the deletion-contraction property directly, or an inductive approach by showing that the property holds on graphs with no edges, and applying induction to the number of edges using deletion-contraction.

  To illustrate the power of this approach, we extend known properties of the chromatic symmetric function on unweighted graphs to the set of vertex-weighted graphs.  In doing so we provide new, alternate proofs of these properties in the unweighted case.

\end{subsection}

\begin{subsection}{$p$-Basis Expansion Formula}

  Given a vertex-weighted graph $(G,w)$, and $S \subseteq E(G)$, we define $\lm(G,w,S)$ to be the partition whose parts are the total weights of the connected components of $(G',w)$, where $G' = (V(G),S)$.
  
  \begin{lemma}
    \begin{equation}\label{eq:pbasis}
      X_{(G,w)} = \sum_{S \subseteq E(G)} (-1)^{|S|} p_{\lm(G,w,S)}
    \end{equation}
  \end{lemma}

  \begin{proof}

    This could be proved by adapting the proof of (\cite{stanley}, Theorem 2.5), but we give a different proof using deletion-contraction.  In our vertex-weighted graph $(G,w)$, let $e_1,e_2,\dots,e_m$ be an ordering of the edges.  We expand $X_{(G,w)}$ in the following manner:  First, in step 1 we apply deletion-contraction to $e_1$, and get

$$
X_{(G,w)} = X_{(G \bk e_1, w)} - X_{(G / e_1, w / e_1)}.
$$
Then, in step 2, we apply deletion-contraction to both of $(G \bk e_1, e)$ and $(G / e_1, w / e_1)$ using edge $e_2$, and obtain an equation with four terms.  Continuing in this manner, in step $i$ we apply deletion-contraction to all $2^{i-1}$ summands created in the previous step, until after step $m$ we have an equation of the form
$$
X_{(G,w)} = \sum_{S \subseteq E(G)} (-1)^{|S|}X_{(G(S), w(S))}
$$
where $(G(S),w(S))$ is the graph resulting from contracting the edges in $S$ and deleting the edges in $E(G) \bk S$, using our given ordering.  This graph has no edges, and each vertex corresponds to a connected component of the graph $G' = (V(G),S)$, since the vertices have been formed by the contraction of exactly those edges in $S$.  Furthermore, the weights of these vertices are the total weights of the connected components of $(G', w)$, since the weight of a vertex in $(G(S),w(S))$ is the sum of the weights of all the vertices in the corresponding component of $(G', w)$.  We recall that if $\overline{K^{\lm}}$ is the graph with no edges and vertices of weights $\lm_1 \geq \dots \geq \lm_k$, then $X_{\overline{K^{\lm}}} = p_{\lm}$.  Thus $X_{(G(S),w(S))} = p_{\lm(G,w,S)}$, and the result follows.

  \end{proof}

\end{subsection}

\begin{subsection}{The Effect of the Symmetric Function Involution}

  Given a graph $G$, define an \emph{orientation of} $G$ to be an assignment of an order (or orientation) to the endpoints of each edge $e \in E(G)$.  If we orient the edge $v_1v_2$ by placing $v_1$ before $v_2$, we write $v_1 \rightarrow v_2$, and say that $v_1$ is the \emph{tail}, and $v_2$ the \emph{head}.  An \emph{oriented cycle} of an orientation of $G$ is a sequence of edges $v_1v_2, v_2v_3,\dots,v_{m-1}v_0, v_0v_1$ that forms a cycle in $G$, and such that these edges are either all oriented such that $v_i \rightarrow v_{i+1}$ for all $i$, or $v_{i+1} \rightarrow v_i$ for all $i$ (with indices taken mod $m$ in both cases).  An \emph{acyclic orientation} of $G$ is one which contains no oriented cycle.

  Recall that with respect to an edge $e = v_1v_2 \in E(G)$ that is not a loop, we define the contracted graph $G / e$ to be $G'$ with $V(G') = (V(G) \bk \{v_1,v_2\}) \cup v^*$, and $E(G') = (E(G) \bk E(v_1, v_2)) \cup E(v^*)$, where $E(v_1,v_2)$ is the set of edges with at least one of $v_1$ or $v_2$ as an endpoint, and $E(v^*)$ consists of each edge in $E(v_1,v_2) \bk \{e\}$ with the endpoint $v_1$ and/or $v_2$ replaced with the new vertex $v^*$.  Using this notation, we define the \emph{contraction of the orientation} $\gamma$ \emph{with respect to} $e$ to be the orientation $\gamma_e$ of $G / e$ where
  \begin{itemize}
  \item Edges of $(E(G) \bk E(v_1, v_2))$ are oriented as they are in $\gamma$.
  \item If $vv_i, i \in \{1,2\}$ is an edge of $G$ for $v \neq v_1,v_2$, orient the corresponding edge $vv^*$ of $G / e$ such that $v$ is a head of $vv_i$ in $\gamma$ if and only if $v$ is a head of $vv^*$ in $\gamma_e$.
  \item If $v_iv_j, i,j \in \{1,2\}$ is an edge of $G$ other than $e$, the corresponding edge $v^*v^*$ is oriented trivially as $v^* \rightarrow v^*$.
  \end{itemize}

  Given a vertex-weighted graph $(G,w)$ and an acyclic orientation $\gamma$ of $G$, we define a coloring $\kappa$ of $G$ to be \emph{weakly proper with respect to} $\gamma$ if for every edge $e = v_1v_2$ oriented as $v_1 \rightarrow v_2$ by $\gamma$, we have $\kappa(v_1) \leq \kappa(v_2)$, and in this case we write

  $$
\overline{x}_{\kappa}(G,w,\gamma) = \prod_{v \in V(G)} x_{\kappa(v)}^{w(v)}.
  $$

    For a vertex-weighted graph $(G,w)$, we define its \emph{weak chromatic symmetric function} as
    \begin{equation}\label{eq:weak}
\overline{X}_{(G,w)} = \sum_{(\gamma, \kappa)} \overline{x}_{\kappa}(G,w,\gamma)
    \end{equation}
    where the sum ranges over all ordered pairs $(\gamma, \kappa)$ with $\gamma$ an acyclic orientation of $G$, and $\kappa$ a weakly proper coloring of $G$ with respect to $\gamma$.

    We prove the following formula for the vertex-weighted weak chromatic symmetric function, extending the formula for unweighted graphs given in (\cite{stanley}, Theorem 4.2):  

    \begin{theorem}\label{thm:invol}
      
      Let $(G,w)$ be a vertex-weighted graph with $n$ vertices and with total weight $d$.  Then
      \begin{equation}\label{eq:inv}
        \overline{X}_{(G,w)} = (-1)^{d-n}\omega(X_{(G,w)})
      \end{equation}
      where $\omega$ is the involution on symmetric functions defined by $\omega(p_{\lm}) = (-1)^{|\lm|-l(\lm)}p_{\lm}$.
    
    \end{theorem}

    \begin{proof}

      We proceed by induction on the number of edges of $G$.  In the base case, the graph has no edges, and vertices of weights $\lm_1 \geq \dots \geq \lm_k$, say.  Then $\overline{X}_{(G,w)} = X_{(G,w)} = p_{\lm}$, where $\lm = (\lm_1,\dots,\lm_k)$, and since $\omega(p_{\lm}) = (-1)^{|\lm|-l(\lm)}p_{\lm} = (-1)^{d-n}p_{\lm}$, the result follows.

      For the inductive step, we consider $(G,w)$ where $G$ has $m \geq 1$ edges, and assume that \eqref{eq:inv} holds for graphs with $m-1$ or fewer edges.  Let $e = v_1v_2$ be an edge of $G$.  Then from the deletion-contraction relation \eqref{eq:delcon}, we deduce that
      \begin{equation}\label{eq:dn}
(-1)^{d-n}\omega(X_{(G,w)}) = (-1)^{d-n}\omega(X_{(G \bk e, w)}) + (-1)^{d-n-1}\omega(X_{(G / e, w / e)}).
      \end{equation}
      By applying the inductive hypothesis to $(G \bk e, w)$ and $(G / e, w / e)$, it suffices to show that
      \begin{equation}\label{eq:invori}
\overline{X}_{(G,w)} = \overline{X}_{(G \bk e, w)} + \overline{X}_{(G / e, w / e)}.
      \end{equation}

      We extend the definition of $\overline{x}_{\kappa}(G,w,\gamma)$ to include all orientations $\gamma$ and all colorings $\kappa$ by defining that $\overline{x}_{\kappa}(G,w,\gamma) = 0$ if $\gamma$ is not acyclic, or if $\kappa$ is not a weakly proper coloring of $G$ with respect to $\gamma$.  Given an orientation $\gamma$ and coloring $\kappa$ on $(G \bk e, w)$, we also define the following:
      \begin{itemize}\label{it:acyc}
      \item $\gamma_1$ is the orientation of $(G,w)$ with $v_1 \rightarrow v_2$ and all other edges oriented as in $\gamma$.
      \item $\gamma_2$ is the orientation of $(G,w)$ with $v_2 \rightarrow v_1$ and all other edges oriented as in $\gamma$.
      \item If $\kappa(v_1) = \kappa(v_2)$, $\kappa_e$ is the coloring of $(G / e, w / e)$ with $\kappa_e(v^*) = \kappa(v_1)$ where $v^*$ is the vertex created by the contraction of $e$, and for all other vertices $v$, $\kappa_e(v) = \kappa(v)$.
      \item If $\kappa(v_1) \neq \kappa(v_2)$, then $\kappa_e$ does not exist (and so $\overline{x}_{\kappa_e}(G / e, w / e, \gamma_e) = 0$).
      \end{itemize}
      Using these definitions, to show \eqref{eq:invori} it suffices to show the stronger statement that for every acyclic orientation $\gamma$ of $G \bk e$, and every weakly proper coloring $\kappa$ of $(G \bk e, w)$ with respect to $\gamma$, we have
      \begin{equation}\label{eq:invkappa}
\overline{x}_{\kappa}(G,w,\gamma_1) + \overline{x}_{\kappa}(G,w,\gamma_2) = \overline{x}_{\kappa}(G \bk e,w,\gamma) + \overline{x}_{\kappa_e}(G / e,w / e,\gamma_e)
      \end{equation}
      since every summand of $\overline{X}_{(G,w)}$, $\overline{X}_{(G / e, w / e)}$ and $\overline{X}_{(G \bk e, w)}$ is counted exactly once in this way.  Note that each of $\overline{x}_{\kappa}(G,w,\gamma_1)$, $\overline{x}_{\kappa}(G,w,\gamma_2)$, $ \overline{x}_{\kappa}(G \bk e,w,\gamma)$, and $ \overline{x}_{\kappa_e}(G / e,w,\gamma_e)$ is either zero or equal to $\overline{x}_{\kappa}(G \bk e,w,\gamma)$, so it is enough to show that the same number of summands on both sides of \eqref{eq:invkappa} are nonzero.
      
      We split into cases based on whether $\gamma$ has a directed path between $v_1$ and $v_2$ (note that it does not contain both a path from $v_1$ to $v_2$ and one from $v_2$ to $v_1$ since then $\gamma$ would contain an oriented cycle).  Suppose for a contradiction that $\gamma$ contains such a path; without loss of generality we may assume it is from $v_1$ to $v_2$.  Then $\gamma_2$ and $\gamma_e$ both contain oriented cycles in their respective graphs.  However, $\gamma_1$ does not contain an oriented cycle in $(G, w)$.  Furthermore, $\kappa(v_1) \leq \kappa(v_2)$ since $\kappa$ is proper with respect to $\gamma$ in $(G \bk e, w)$ and there is a directed path from $v_1$ to $v_2$, so $\kappa$ is proper with respect to $\gamma_1$ in $(G, w)$.  Thus, \eqref{eq:invkappa} holds in this case.

      Now assume that there is no directed path.  Then all of $\gamma_1$, $\gamma_2$, and $\gamma_e$ are acyclic orientations.  We split into subcases based on $\kappa$.  If $\kappa(v_1) = \kappa(v_2)$, then $\kappa_e$ exists and is proper with respect to $\gamma_e$, and $\kappa$ is proper with respect to all of $\gamma$, $\gamma_1$, and $\gamma_2$, so \eqref{eq:invkappa} holds.  Otherwise, without loss of generality suppose that $\kappa(v_1) < \kappa(v_2)$.  Then $\kappa_e$ does not exist, and $\kappa$ is not proper with respect to $\gamma_2$, but $\kappa$ is proper with respect to $\gamma_1$, so \eqref{eq:invkappa} also holds in this case.  This concludes the proof.
    \end{proof}

    As a corollary, we deduce a further result about the function $\overline{X}_{(G,w)}$ that extends the corresponding result on unweighted graphs from (\cite{stanley}, Theorem 2.7):

    \begin{cor}
      If $(G,w)$ is a vertex-weighted graph with $n$ vertices and total weight $d$, then 
    \begin{equation}\label{eq:ppositive}
    \overline{X}_{(G,w)} = (-1)^{d-n}\omega(X_{(G,w)}) 
    \end{equation}
    is $p$-positive.
    \end{cor}

    \begin{proof}

      We proceed by induction on the number of edges.  The base case is a graph with no edges, and as was noted at the beginning of the previous proof, if such a graph $(G,w)$ has vertices of weights $\lm_1 \geq \dots \geq \lm_k$ say, then $\overline{X}_{(G,w)} = p_{\lm}$ where $\lm = (\lm_1,\dots,\lm_k)$, and this is $p$-positive.

      For the inductive step, suppose that $(G,w)$ has $m \geq 1$ edges, and suppose that we have shown that the claim holds for vertex-weighted graphs $(G,w)$ with $m-1$ edges.  Then for any edge $e \in E(G)$, using the inductive hypothesis and the relation \eqref{eq:invori} shows that $\overline{X}_{(G,w)}$ is a sum of two $p$-positive functions, and hence it is $p$-positive, and this concludes the proof.

    \end{proof}

  \end{subsection}

\begin{subsection}{A Formula on Cycles}

  We now prove a modular relation on cycles that was originally proved for unweighted graphs by (\cite{dahl}, Proposition 5):

  \begin{theorem}\label{thm:cycles}

Let $(G, w)$ be a vertex-weighted graph containing a cycle $C$, and let $e$ be a fixed edge of this cycle.  Then 
\begin{equation}\label{eq:cycle}
\sum_{S \subseteq E(C) \bk e} (-1)^{|S|} X_{(G \bk S,w)} = 0.
\end{equation}

  \end{theorem}

  \begin{proof}

    We proceed by induction on the number of edges in the cycle.  The base case of a $1$-edge cycle (a loop) is immediate.

For the inductive step, we assume the claim holds for graphs with an $n$-edge cycle and show that it holds on graphs with an $(n+1)$-edge cycle.  Let $(G,w)$ be a vertex-weighted graph with an $(n+1)$-edge cycle $C$, let $e$ be the edge in the statement of Theorem \ref{thm:cycles}, and let $f = v_1v_2$ be an edge of the cycle with $e \neq f$.  We apply deletion-contraction to the edge $f$ to get
$$X_{(G,w)} = X_{(G \bk f, w)} - X_{(G / f, w / f)}.$$

Let $v^*$ be the vertex of $G / f$ formed by the contraction of $v_1$ and $v_2$.  We now apply the inductive hypothesis to $(G / f, w / f)$, since in this graph $C / f$ is an $n$-edge cycle containing the edge $e$.  We obtain
$$X_{(G,w)} =   X_{(G \bk f, w)} - \sum_{\emptyset \neq S' \subseteq E(C / f) \bk e} (-1)^{|S'|-1} X_{((G / f) \bk S', w / f)}. $$ 

Now, in this sum, for every summand we will uncontract the graph $((G / f) \bk S', w / f)$ to $(G \bk S',w)$, thus obtaining

\begin{align*}
  X_{(G,w)} &= \quad X_{(G \bk f, w)} \\  &- \sum_{\emptyset \neq S' \subseteq E(C) \bk \{e,f\}} (-1)^{|S'|-1} X_{(G \bk (S' \cup \{f\}), w)} \\ &+ \sum_{\emptyset \neq S' \subseteq E(C )\bk \{e,f\}} (-1)^{|S'|-1}X_{(G \bk S', w)}. \stepcounter{equation}\tag{\theequation}\label{eq:S1}
\end{align*}

We claim that the right-hand side of this equation is equal to
\begin{equation}\label{eq:S2}
  \sum_{\emptyset \neq S \subseteq E(C) \bk e} (-1)^{|S|-1} X_{(G \bk S, w)}
\end{equation}
which is sufficient to complete the proof.

The term $X_{(G \bk f,w)}$ of \eqref{eq:S1} is equal to the term of \eqref{eq:S2} corresponding to $S = \{f\}$.  The subtracted sum 
$$-\sum_{\emptyset \subsetneq S' \subseteq E(C) \bk \{e,f\}} (-1)^{|S'|-1} X_{(G \bk (S' \cup \{f\}), w)}$$
in \eqref{eq:S1} is equal to the sum of those terms of \eqref{eq:S2} corresponding to sets $S = \{f\} \cup S'$ with $S' \neq \emptyset$.  Finally, the sum
$$
\sum_{\emptyset \subsetneq S' \subseteq E(C )\bk \{e,f\}} (-1)^{|S'|-1}X_{(G \bk S', w)}
$$
of \eqref{eq:S1} is equal to the sum of the terms of \eqref{eq:S2} corresponding to sets $S = S'$ where $S'$ is a nonempty subset of $C \backslash \{e,f\}$.

  \end{proof}

  \end{subsection}

\end{section}

\begin{section}{Acyclic Orientations}\label{acyc}

Let $a(G)$ denote the number of acyclic orientations of a graph $G$. In terms of deletion-contraction, for any edge $e \in E(G)$ that is not a loop, it is easy to check that
  \begin{equation}\label{eq:sumacyc}
    a(G) = a(G \backslash e) + a(G / e)
  \end{equation}
It can be shown, either by using \eqref{eq:sumacyc} and induction, or using a chromatic polynomial version of \eqref{eq:invori} as in \cite{acyc}, that if $G$ is a graph on $n$ vertices, then 

  \begin{equation}\label{eq:chi}
    a(G) = (-1)^n\chi_G(-1)
  \end{equation}
Additionally, if $\gamma$ is an orientation of a graph $G$, we call a vertex $v \in V(G)$ a \emph{sink} of $\gamma$ if $v$ is not the tail of any edge of $\gamma$.  Then \eqref{eq:chi} is generalized by the following theorem:

  \begin{theorem}{(\cite{stanley}, Theorem 3.3)}

  Let $G$ be an unweighted graph.  We write its chromatic symmetric function in the elementary symmetric function basis as 
$$X_G = \sum_{\lm} c_{\lm}e_{\lm}.$$
Then 
\begin{equation}\label{eq:sinkmain}
a(G) = \sum_{\lm} c_{\lm}.
\end{equation}
Furthermore, as a refinement, define $a_m(G)$ to be the number of acyclic orientations of $G$ having exactly $m$ sinks.  Then
\begin{equation}\label{eq:sinks}
  a_m(G) = \sum_{l(\lm) = m} c_{\lm}.
\end{equation}
That is, $a_m(G)$ is given by the sum of those $c_{\lm}$ corresponding to partitions $\lm$ with exactly $m$ parts.

\end{theorem}

Notably, the proof method of \cite{stanley} uses a novel algebraic argument that does not generalize directly either the argument of \cite{acyc} or the inductive method suggested by \eqref{eq:sumacyc}.

We will prove a generalization for vertex-weighted graphs using induction and the deletion-contraction relation.  In this way, we also provide an alternate proof of \eqref{eq:sinks} that is a natural extension of enumerative proofs of \eqref{eq:chi}.

We first establish some notation and terminology.  For a symmetric function $f$, if $f = \sum_{\lm} c_{\lm}e_{\lm}$ is its expansion in the basis of elementary symmetric functions, we define $\sigma(f) = \sum_{\lm} c_{\lm}$, and $\sigma_m(f) = \sum_{l(\lm) = m} c_{\lm}$.

For an acyclic orientation $\gamma$ of a vertex-weighted graph $(G,w)$, we define $Sink(\gamma)$ to be the set of sinks of $G$ with respect to $\gamma$ (note that as $\gamma$ is acyclic, $Sink(\gamma)$ is always nonempty).  Let $sink(\gamma) = |Sink(\gamma)|$.  Define a \emph{sink map} $S$ \emph{of} $\gamma$ to be a function $S : Sink(\gamma) \rightarrow 2^{\mathbb{N}}$ such that for all $v \in Sink(\gamma)$, $\emptyset \neq S(v) \subseteq \{1,2,\dots,w(v)\}$.  Given a sink map $S$ of an acyclic orientation $\gamma$ on a vertex-weighted graph $(G,w)$, we define its \emph{sink weight} to be $swt(G,w,\gamma,S) = \sum_{v \in Sink(\gamma)} |S(v)|$.  When $(G,w)$ and/or $\gamma$ are clear from context we may use $swt(S)$ or $swt(\gamma, S)$ in place of $swt(G,w,\gamma,S)$ for brevity.

We now state the main theorem of this section:

\begin{theorem}\label{thm:elm}

  Let $(G,w)$ be a vertex-weighted graph with $n$ vertices and total weight $d$.  Then
  $$
 \sigma(X_{(G,w)}) = (-1)^{d-n}\sum_{(\gamma,S)} (-1)^{swt(S)-sink(\gamma)}
 $$
 where the sum runs over all ordered pairs $(\gamma,S)$ such that $\gamma$ is an acyclic orientation of $(G,w)$, and $S$ is a sink map of $\gamma$.  Additionally,
\begin{equation}\label{eq:acyc}
  \sigma_m(X_{(G,w)}) = (-1)^{d-n}\sum_{swt(\gamma, S)=m} (-1)^{m-sink(\gamma)}
\end{equation}
where the sum ranges only over those ordered pairs $(\gamma, S)$ with $swt(S) = m$.

\end{theorem}

\begin{proof}

  It suffices to prove \eqref{eq:acyc}.  We proceed by induction on the number of edges of $(G,w)$.  The base case is a vertex-weighted graph with no edges.  If such a graph has vertices of weights $\lm_1 \geq \dots \geq \lm_k$, then $X_{(G,w)} = p_{\lm}$ where $\lm = (\lm_1,\dots,\lm_k)$.

  First, we establish the following identity for any positive integer $a$:
  \begin{equation}\label{eq:minibinom}
  \sigma_m(p_a) = (-1)^{a-m}\binom{a}{m}
  \end{equation}

  We show this for fixed $a$ by induction on $m$, making use of Newton's identity (\cite{mac}, Chapter 1.2):
  $$
  p_a = (-1)^{a-1}ae_a + \sum_{i=1}^{a-1} (-1)^{a-1+i}e_{a-i}p_i.
  $$
  The case $m = 1$ is clear from this.  Now we assume the claim holds for $m = b-1$ and prove it for $m = b$.  Using Newton's identity followed by the inductive hypothesis we have
  \begin{align*}
    \sigma_b(p_a) &= \sum_{i=1}^{a-1} (-1)^{a-1+i}\sigma_{b-1}(p_i) = \sum_{i=1}^{a-1} (-1)^{a-1+i}(-1)^{i-b+1}\binom{i}{b-1} \\
    &= (-1)^{a-b}\sum_{i=1}^{a-1}\binom{i}{b-1} = (-1)^{a-b}\binom{a}{b}
  \end{align*}
  where we have used the Hockey Stick Identity.

   We now establish the base case for the induction of the main proof.  Recall that $(G,w)$ has vertices of weights $\lm_1 \geq \dots \geq \lm_k$ and no edges.  First we evaluate directly the left-hand side of \eqref{eq:acyc}:
  $$
  \sigma_m(X_{(G,w)}) = \sigma_m(p_{(\lm_1,\dots,\lm_k)}) = \sum_{(a_1,\dots,a_k)} \sigma_{a_1}(p_{\lm_1}) \cdots \sigma_{a_k}(p_{\lm_k})
  $$
  where this sum runs over all tuples $(a_1,\dots,a_k)$ of positive integers satisfying $a_i \leq \lm_i$ and \\ $a_1 + \cdots + a_k = m$.  Expanding using \eqref{eq:minibinom}, we get
  \begin{equation}\label{eq:binom}
  \sum_{(a_1,\dots,a_k)} \prod_{i=1}^k (-1)^{\lm_i-a_i}\binom{\lm_i}{a_i} = \sum_{(a_1,\dots,a_k)} (-1)^{|\lm|-m} \prod_{i=1}^k \binom{\lm_i}{a_i}.
  \end{equation}

  Next we will simplify the right-hand side of \eqref{eq:acyc} and show that it is equal to \eqref{eq:binom}.  In $(G,w)$ there is only one acyclic orientation $\gamma$, the empty orientation, and all vertices are sinks in this orientation, so equivalently we are looking for all ways to choose the sink map $S$ such that $swt(S) = m$.  Then the sum simplifies to 
  $$
(-1)^{d-n}\sum_{swt(\gamma, S)=m} (-1)^{m-sink(\gamma)} = (-1)^{|\lm|-k}(-1)^{m-k}\sum_{(a_1,\dots,a_k)} \prod_{i=1}^k \binom{\lm_i}{a_i}
    $$
  where the sum runs over the same tuples as in \eqref{eq:binom}.  Clearly these sums are equal, and this establishes the base case.

  We now show the inductive step.  Let $(G,w)$ be a vertex-weighted graph with $g \geq 1$ edges, and assume that \eqref{eq:acyc} holds for all vertex-weighted graphs with fewer than $g$ edges.  We may assume that $(G,w)$ has no loops, as otherwise both sides of \eqref{eq:acyc} are $0$.  Let $e$ be an edge of $(G,w)$, with endpoints $v_1$ and $v_2$.  In $(G / e, w / e)$, let $v^*$ be the vertex arising from the contraction of $v_1$ and $v_2$.  Taking the deletion-contraction relation \eqref{eq:edgeadd}, applying $\sigma_m$ to both sides, and multiplying both sides by $(-1)^{d-n}$ we have
  $$
(-1)^{d-n}\sigma_m(X_{(G \backslash e,w)}) = (-1)^{d-n}\sigma_m(X_{(G,w)}) - (-1)^{d-n-1}\sigma_m(X_{(G / e, w / e)}).
  $$
  By the inductive hypothesis
  $$
(-1)^{d-n}\sigma_m(X_{(G \backslash e,w)}) = \sum_{swt(G \bk e,\gamma,S)=m} (-1)^{m-sink(\gamma)}
  $$
  and
  $$
(-1)^{d-n-1}\sigma_m(X_{(G / e, w / e)}) = \sum_{swt(G/e,\gamma,S)=m} (-1)^{m-sink(\gamma)}.
  $$
  To finish the proof it suffices to show that
  $$
\sum_{swt(G \bk e, \gamma,S)=m} (-1)^{m-sink(\gamma)} = \sum_{swt(G,\gamma,S)=m} (-1)^{m-sink(\gamma)} - \sum_{swt(G/e,\gamma,S)=m} (-1)^{m-sink(\gamma)}
$$
or after multiplying both sides by $(-1)^m$,
\begin{equation}\label{eq:sink}
 \sum_{swt(G \bk e,\gamma,S)=m} (-1)^{sink(\gamma)} = \sum_{swt(G,\gamma,S)=m} (-1)^{sink(\gamma)} - \sum_{swt(G/e,\gamma,S)=m} (-1)^{sink(\gamma)}.
\end{equation}

To prove \eqref{eq:sink}, we will work over a larger class of maps $S$ whose domain is the set of all vertices of a graph instead of just the sinks of a given acyclic orientation $\gamma$, and we also allow $S(v) = \emptyset$ for all vertices $v$, while still requiring that $S(v) \subseteq \{1,2,\dots,w(v)\}$.  We call $S$ \emph{$\gamma$-admissible} if $S(v) \neq \emptyset$ if and only if $v \in Sink(\gamma)$.  Thus we may rephrase \eqref{eq:sink} by allowing $\gamma$ and $S$ in the summations to range over all acyclic orientations $\gamma$ and all sink maps $S$ with $S(v) \subseteq \{1,2,\dots,w(v)\}$ for all $v$, but where we define the corresponding summand to be $(-1)^{sink(\gamma)}$ if and only if $S$ is $\gamma$-admissible, and $0$ otherwise.

We show that for every acyclic orientation $\gamma_0$ of $G \backslash e$, and every map $S_0 : V(G) \rightarrow 2^{\mathbb{N}}$ with $S_0(v) \subseteq \{1,2,\dots,w(v)\}$ such that $\sum_{v \in V(G)} |S_0(v)| = m$, the equation \eqref{eq:sink} is satisfied when summing over those $\gamma$ and $S$ where
 \begin{itemize}
 \item in $G \backslash e$, $\gamma = \gamma_0$ and $S = S_0$.
 \item in $G$, $\gamma$ restricted to $G \backslash e$ is $\gamma_0$, and $S = S_0$.  This yields two choices for $\gamma$ depending on the orientation of the edge $v_1v_2$.  Let $\gamma_{v_1}$ be the one where $v_1 \rightarrow v_2$, and $\gamma_{v_2}$ the one where $v_2 \rightarrow v_1$.
 \item in $G / e$, $\gamma = \gamma_{v^*}$ is the contraction of $\gamma_0$, and $S = S'$ is defined by $S'(v) = S_0(v)$ if $v \neq v^*$, and $S'(v^*) = S(v_1) \cup \{w(v_1)+i : i \in S(v_2)\}$.
 \end{itemize}

 It is easy to check that every pair $(\gamma,S)$ for each of $G \backslash e$, $G$, and $G / e$ is derived from exactly one such $(\gamma_0,S_0)$, so proving this claim will finish the proof of the theorem.

 For ease of notation, we fix $\gamma_0$ and $S_0$ in what follows.   Let $T(G \bk e)$ denote the term corresponding to $\gamma_0$ and $S_0$ in the summation for $G \bk e$ in \eqref{eq:sink}, and likewise let $T(G / e)$ denote the term in the summation for $G / e$ corresponding to $\gamma_{v_*}$ and $S'$.  Let $T(G_{v_1})$ denote the term in the summation for $G$ corresponding to $S_0$ and $\gamma_{v_1}$, and likewise for $T(G_{v_2})$.  Thus what we must show for every fixed $\gamma_0$ and $S_0$ is
 \begin{equation}\label{eq:terms}
T(G \bk e) = T(G_{v_1}) + T(G_{v_2}) - T(G / e)
 \end{equation}

 We proceed by cases:
 
 \vspace{0.5cm}
 
 \textbf{Case 1:  $\gamma_0$ has a directed path from $v_1$ to $v_2$ or from $v_2$ to $v_1$}

 \vspace{0.5cm}
 
 Note that $\gamma_0$ cannot have both of these directed paths since $\gamma_0$ is acyclic.  Without loss of generality we assume the path is from $v_2$ to $v_1$. Then $T(G / e) = 0$ because $\gamma_{v^*}$ is not acyclic.  Also $T(G \bk e) = (-1)^{sink(\gamma_0)}$ if $S_0$ is $\gamma_0$-admissible, and $0$ otherwise.  The orientation $\gamma_{v_1}$ is not acyclic, so $T(G_{v_1}) = 0$.  However, $\gamma_{v_2}$ is acyclic, and has the same set of sinks as $\gamma_0$, so also $T(G_{v_2}) = (-1)^{sink(\gamma_0)}$ if $S_0$ is $\gamma_0$-admissible and $0$ otherwise, and this satisfies \eqref{eq:terms}.

\vspace{0.5cm}

 Note that from now on, since we may assume there is no directed path between $v_1$ and $v_2$ in $\gamma_0$, the orientation $\gamma_{v^*}$ is acyclic.

 \vspace*{0.5cm}

 \textbf{Case 2: Neither $v_1$ nor $v_2$ is a sink with respect to $\gamma_0$}

 \vspace{0.5cm}

 In this case, $v_1$ and $v_2$ are also not sinks in $\gamma_{v_1}$ or $\gamma_{v_2}$, and $v^*$ is not a sink in $\gamma_{v^*}$, so if it is not the case that $S_0(v_1) = S_0(v_2) = \emptyset$, all terms of \eqref{eq:terms} are equal to $0$.  Otherwise, all terms are equal to $1$.  In either case, \eqref{eq:terms} is satisfied.

 \vspace{0.5cm}

 \textbf{Case 3: Exactly one of $v_1$ or $v_2$ is a sink with respect to $\gamma_0$}

 \vspace{0.5cm}

 Without loss of generality we may assume that $v_1$ is a sink; the case where $v_2$ is a sink is analogous.  Similarly to the previous case, if $S_0(v_2) \neq \emptyset$ then all terms of \eqref{eq:terms} are equal to $0$ (note that in $\gamma_{v^*}$, vertex $v^*$ is a sink if and only if both $v_1$ and $v_2$ are).  Thus, we may assume that $S_0(v_2) = \emptyset$.  We have two subcases to consider:

 \vspace{0.5cm}

 \textbf{Case 3.1: $S_0(v_1) = \emptyset$}

 \vspace{0.5cm}

 In this case $T(G \bk e) = 0$ as $S_0$ is not $\gamma_0$-admissible.  Additionally, $T(G_{v_2}) = 0$ as $S_0$ is not $\gamma_{v_2}$-admissible.  However, $S_0$ is $\gamma_{v_1}$-admissible since $v_1$ is no longer a sink in $\gamma_{v_1}$, so $T(G_{v_1}) = (-1)^{sink(\gamma_0)-1}$.  Also, as $v^*$ is not a sink in $\gamma_{v^*}$, we have $T(G / e) =  (-1)^{sink(\gamma_0)-1}$.  Thus, this case satisfies \eqref{eq:terms}.

 \vspace{0.5cm}

 \textbf{Case 3.2: $S_0(v_1) \neq \emptyset$}

 \vspace{0.5cm}

 Then $S_0$ is $\gamma_0$-admissible, so $T(G \bk e) = (-1)^{sink(\gamma_0)}$.  Also $S_0$ is $\gamma_{v_2}$-admissible, but it is not $\gamma_{v_1}$-admissible, so $T(G_{v_2}) = (-1)^{sink(\gamma_0)}$ and $T(G_{v_1}) = 0$.  The map $S'$ is not $\gamma_{v^*}$-admissible, since $v^*$ is not a sink, but $S'(v^*) \neq \emptyset$ since $S_0(u) \neq \emptyset$, so $T(G / e) = 0$.  This satisfies \eqref{eq:terms}.

  \vspace{0.5cm}

 \textbf{Case 4: Both $v_1$ and $v_2$ are sinks with respect to $\gamma_0$}

 \vspace{0.5cm}

 If $S_0(v_1) = S_0(v_2) = \emptyset$, then all terms of \eqref{eq:terms} are equal to $0$.  Thus we may assume that at least one of these sets is nonempty.  We again split into subcases.

 \vspace{0.5cm}

 \textbf{Case 4.1: Exactly one of $S_0(v_1)$ and $S_0(v_2)$ is nonempty}

 \vspace{0.5cm}

 Without loss of generality we may assume that $S_0(v_1) \neq \emptyset$ and $S_0(v_2) = \emptyset$; the other case is analogous.  Then $S_0$ is not $\gamma_0$-admissible, so $T(G \bk e) = 0$.  Also, $S_0$ is not $\gamma_{v_1}$-admissible, so $T(G_{v_1}) = 0$.  However, $S_0$ is $\gamma_{v_2}$-admissible since here $v_2$ is no longer a sink, so $T(G_{v_2}) = (-1)^{sink(\gamma_0)-1}$.  In $\gamma_{v^*}$, the contracted vertex $v^*$ is a sink and $S'(v^*)$ is nonempty, so $T(G / e) = (-1)^{sink(\gamma_0)-1}$ since the two sinks $v_1$ and $v_2$ became one sink.  This satisfies \eqref{eq:terms}.

 \vspace{0.5cm}

 \textbf{Case 4.2: Both $S_0(v_1)$ and $S_0(v_2)$ are nonempty}

 \vspace{0.5cm}

 In this case $S_0$ is $\gamma_0$-admissible, so $T(G \bk e) = (-1)^{sink(\gamma_0)}$.  However, $S_0$ is neither $\gamma_{v_1}$-admissible nor $\gamma_{v_2}$-admissible, so $T(G_{v_1}) = T(G_{v_2}) = 0$.  In $\gamma_{v^*}$, the contracted vertex $v^*$ is a sink, and $S'(v^*)$ is nonempty, so  $T(G / e) = (-1)^{sink(\gamma_0)-1}$, since the two sinks $v_1$ and $v_2$ became one sink.  This satisfies \eqref{eq:terms}.

 \vspace{0.5cm}

 Thus we have shown that \eqref{eq:terms} holds in all cases, and this finishes the proof. 
 
\end{proof}

\end{section}

\begin{section}{Further Applications}

  Considering vertex-weighted graphs with the chromatic symmetric function provides a new perspective and new tools for approaching major unsolved problems.  We mention some of these problems and possible approaches.

\begin{subsection}{Chromatic Quasisymmetric Functions}

  One well-researched generalization of the chromatic symmetric function is the chromatic quasisymmetric function of Shareshian and Wachs \cite{wachs}, defined on vertex-labeled graphs, or equivalently on graphs equipped with an acyclic orientation.  In the context of finding a deletion-contraction relation it is more natural to look at the generalization of this function to simple graphs with an arbitrary orientation, considered by Ellzey \cite{ell} and Alexandersson and Panova \cite{per}.  Given a graph $G$ with a fixed orientation $\gamma$, for any proper coloring $\kappa$ of $G$ define the \emph{ascent number} $asc(\kappa)$ to be the number of edges $v_1 \rightarrow v_2$ of $\gamma$ such that $\kappa(v_1) < \kappa(v_2)$.  Using the notation $x_{\kappa}(G) = \prod_{v \in V(G)} x_{\kappa(v)}$, define the \emph{chromatic quasisymmetric function of} $G$ \emph{with respect to} $\gamma$ as

  \begin{equation}\label{eq:quasi}
X_{(G,\gamma)}(q,x_1,x_2,\dots) = \sum_{\kappa} x_{\kappa}(G)q^{asc(\kappa)}
  \end{equation}
  where the sum runs over all proper colorings $\kappa$ of $G$.

  It is natural to try to extend our definition on vertex-weighted graphs to work in this setting. Ideally, such an extension would equip the chromatic quasisymmetric function with a deletion-contraction relation.  However, a first attempt
\begin{equation}\label{eq:delconquasi}
X_{(G,w,\gamma)}(q,x_1,x_2,\dots) = \sum_{\kappa} x_{\kappa}(G,w)q^{asc(\kappa)}
\end{equation}
fails to provide a deletion-contraction relation.  To see this, consider a vertex-weighted graph $(G,w)$ with edge $e = v_1v_2$, and assume we are considering an orientation of $G$ in which $v_1 \rightarrow v_2$.  In the following table, we determine how the power of $q$ in a proper coloring $\kappa$ of $G \bk e$ relates to the power of $q$ of corresponding proper colorings of $G / e$ or $G$ (thus all numbers $asc(\kappa)$ are relative to $G \bk e$):

\begin{center}
  \begin{tabular}{|c|c|c|c|}
    \hline
    $\kappa(v_1)$ vs $\kappa(v_2)$ & $G \bk e$ & $G / e$ & $G$ \\
    \hline
    $=$ & $q^{asc(\kappa)}$ & $q^{asc(\kappa)}$ & N/A \\
    \hline
    $>$ & $q^{asc(\kappa)}$ & N/A & $q^{asc(\kappa)}$ \\
    \hline
    $<$ & $q^{asc(\kappa)}$ & N/A & $q^{asc(\kappa)+1}$ \\
    \hline
  \end{tabular}
\end{center}

An explicit example illustrates the problem implied by the asymmetry of this table.  Consider the three-vertex path $P_3$ with orientation $\gamma_0$ that has exactly one sink, and exactly one source (a vertex that is not the head of any oriented edge), and with all vertex weights equal to $1$.  Then if we consider powers of $q$ in the specialization $x_1 = x_2 = x_3 = 1$, $x_i = 0$ for $i \geq 4$, we have

$$
X_{(P_3,1^3,\gamma_0)}(q,1,1,1,0,0,\dots) = q^2+10q+1.
$$
If we take the edge $e$ of this path that has the source as an endpoint, then using the same specialization we have
$$
X_{(P_3 \bk e, 1^3, \gamma_0)} = 9q + 9
$$
and
$$
X_{(P_3 / e, 1^3 / e, \gamma_0 / e)} = 3q + 3.
$$

Note that $q+1$ divides the second and third of these functions but not the first, so we cannot have a simple deletion-contraction relation of the form $X_{(G,w,\gamma)} = (f(q))^aX_{(G \bk e, w, \gamma \bk e)} \pm (g(q))^bX_{(G / e, w / e, \gamma / e)}$ where $f$ and $g$ are polynomials in $q$.

There is a way to provide a deletion-contraction analog if we expand upon how the relation may look.  Given a vertex-weighted graph $(G,w)$ with fixed edge $e = v_1v_2$, and an orientation $\gamma$, define $\gamma_{\overleftarrow{e}}$ to be the same orientation as $\gamma$ except with the order of the head and tail of edge $e$ reversed.  Then the following relation holds:

\begin{lemma}
  Let $(G,w)$ be a vertex-weighted graph, and let $e$ be an edge of $G$.  Let $\gamma$ be an orientation of $G$.  Then
  \begin{equation}\label{eq:delconflip}
X_{(G,w,\gamma)} + X_{(G,w,\gamma_{\overleftarrow{e}})} = (1+q)(X_{(G \bk e, w, \gamma)} - X_{(G / e, w / e, \gamma / e)}).
  \end{equation}
\end{lemma}

\begin{proof}
  We first rearrange \eqref{eq:delconflip} into
  \begin{equation}\label{eq:flipplus}
(1+q)X_{(G \bk e, w, \gamma)} = X_{(G,w,\gamma)} + X_{(G,w,\gamma_{\overleftarrow{e}})} + (1+q)X_{(G / e, w / e, \gamma / e)}.
  \end{equation}

  The result is clear if $e$ is a loop, so let $e = v_1v_2$ have distinct endpoints.  We show that \eqref{eq:flipplus} holds by showing a one-to-one correspondence between the terms on the left-hand side, and sets of terms on the right-hand side.  Consider a proper coloring $\kappa$ of $G \bk e$, and let $asc(\kappa)$ be the ascent number of $\kappa$ with respect to $\gamma$.  We split into cases based on the colors $\kappa$ gives to $v_1$ and $v_2$.

  If $\kappa(v_1) = \kappa(v_2)$, then the term $(1+q)x_{\kappa}(G \bk e, w)q^{asc(\kappa)}$ is equal to the term \\ $(1+q)x_{\kappa^e}(G / e, w / e)q^{asc(\kappa)}$ of $(1+q)X_{(G / e, w / e, \gamma / e)}$, where $\kappa^e$ is the contraction of the coloring $\kappa$ with respect to $e$.  Furthermore, there is no corresponding term of either $X_{(G,w,\gamma)}$ or $X_{(G,w,\gamma_{\overleftarrow{e}})}$ since $\kappa$ is not a proper coloring of $G$.

  If $\kappa(v_1) \neq \kappa(v_2)$, then we get a term of $(1+q)x_{\kappa}(G \bk e, w)q^{asc(\kappa)}$ on the left-hand side of \eqref{eq:flipplus} as before.  There is no corresponding term of $X_{(G / e, w / e, \gamma / e)}$ since the coloring $\kappa$ does not contract to one on $G / e$.  We do get corresponding terms of both $X_{(G,w,\gamma)}$ and $X_{(G,w,\gamma_{\overleftarrow{e}})}$, equal to $x_{\kappa}(G, w)q^{asc(\kappa)}$ and $x_{\kappa}(G, w)q^{asc(\kappa)+1}$  in some order depending on the orientation of $e$.  Furthermore, these terms satisfy
  $$
(1+q)x_{\kappa}(G \bk e, w)q^{asc(\kappa)} = x_{\kappa}(G, w)q^{asc(\kappa)} + x_{\kappa}(G, w)q^{asc(\kappa)+1}
  $$
  since $x_{\kappa}(G \bk e, w) = x_{\kappa}(G, w)$.

  Thus, there is a bijective correspondence of terms from the left-hand side of \eqref{eq:flipplus} to a unique set of terms on the right-hand side, and this concludes the proof.
  
\end{proof}

\end{subsection}

\begin{subsection}{$e$- and $s$-positivity}

Another possible application of the deletion-contraction method is expanding the chromatic symmetric functions of certain families of graphs to prove that the coefficients are nonnegative in a fixed basis.  For this purpose, the edge-addition form \eqref{eq:edgeadd} of the deletion-contraction relation seems promising, as it is a sum instead of a difference, so would maintain nonnegativity.

Research in positivity has generally focused on incomparability graphs of partially ordered sets.  A \emph{partially ordered set} (or \emph{poset}) is pair $(P, \preceq)$ where $P$ is a set, and $\preceq$ is a binary relation on $P$ that is a \emph{partial order} on $P$, meaning that the following are satisfied:

\begin{itemize}
\item $p \preceq p$ for all $p \in P$ (reflexivity)
\item For any distinct $p, q \in P$ at most one of $p \preceq q$ and $q \preceq p$ holds (antisymmetry)
\item For any $p,q,r \in P$, if $p \preceq q$ and $q \preceq r$ then $p \preceq r$ (transitivity)
\end{itemize}
Distinct elements $p, q \in P$ are \emph{comparable} if either $p \preceq q$ or $q \preceq p$; otherwise $p$ and $q$ are \emph{incomparable}.  The \emph{incomparability graph} of the poset $(P, \preceq)$ is the simple graph with vertex set $P$ and edge set $\{pq : \text{p and q are incomparable}\}$.  It is denoted by $Inc(P, \preceq)$, or just $Inc(P)$ if there is no ambiguity in the choice of relation $\preceq$.

Any $Q \subseteq P$ defines an \emph{induced subposet} of $P$ by simply considering $(Q, \preceq)$.  Research on positivity focuses on incomparability graphs of $(3+1)$-free posets,  meaning those $Inc(P)$ where $P$ does not contain an induced subposet isomorphic to $(\{a,b,c,d\}, \preceq)$, where $a \preceq b, b \preceq c, a \preceq c$, and $d$ is incomparable to all of $a,b,c$.  In particular, one of the most important open problems involving the chromatic symmetric function is the Stanley-Stembridge conjecture:
\begin{conj}\label{conj:stanstem}{(\cite{stanley})}
Every incomparability graph of a $(3+1)$-free poset is $e$-positive.\footnote{Notably, \cite{guay} shows that this conjecture is equivalent to the statement that unit interval graphs are $e$-positive.  This is a seemingly stronger statement, since unit interval graphs are precisely the incomparability graphs of posets that are simultaneously $(3+1)-$ and $(2+2)-$free.}
\end{conj}
Thus, we consider $e$-positivity of vertex-weighted graphs in an attempt to approach the Stanley-Stembridge conjecture, perhaps by generalizing it to a class of graphs on which deletion-contraction may be applied.

In the case of a vertex-weighted graph $(G,w)$ with $n$ vertices and total weight $d$, it is easy to see from equations \eqref{eq:pbasis} and \eqref{eq:ppositive} that in the $e$-basis expansion the coefficient of $e_d$ is $(-1)^{d-n}$, so the natural extension is to ask whether $(-1)^{d-n}X_{(G,w)}$ is $e$-positive.

This question of $e$-positivity can be answered for all vertex-weighted graphs with nontrivial vertex weights.  We define a \emph{connected partition} of a vertex-weighted graph $(G,w)$ to be a partition $P_1 \sqcup \dots \sqcup P_m = V(G)$ of the vertex set such that for each $i$, the subgraph of $G$ induced by restricting to $P_i$ is connected.  We define the \emph{type} of a connected partition to be the integer partition whose parts are $w(P_1), \dots ,w(P_m)$.  The following lemma may be proved by a straightforward generalization of the proof of (\cite{wolfgang}, Proposition 1.3.3):

\begin{lemma}\label{lem:epos}

If $(G,w)$ is a vertex-weighted graph with $n$ vertices and total weight $d$ such that $(-1)^{d-n}X_{(G,w)}$ is $e$-positive, and $(G,w)$ has a connected partition of type $\lm \vdash d$, then it also has a connected partition of type $\mu$ for every partition $\mu$ that is a refinement of $\lm$.

\end{lemma}

This yields:

\begin{cor}\label{cor:ebasis}

  Let $(G,w)$ be a vertex-weighted graph.  If there is a vertex $v \in V(G)$ such that $w(v) > 1$, then $(-1)^{d-n}X_{(G,w)}$ is not $e$-positive.

\end{cor}

\begin{proof}

Let $(G,w)$ have vertex weights $w_1 \geq \dots \geq w_n$.  Then $(G,w)$ has a connected partition of type $(w_1,\dots,w_n)$, but it does not have one of type $1^d$, so by the previous lemma, $(-1)^{d-n}X_{(G,w)}$ is not $e$-positive.

\end{proof}

Although this answers the natural weighted analogue of the Stanley-Stembridge Conjecture, there is more work to be done here.  The $e$-basis may not be optimal for considering vertex-weighted graphs; perhaps there is a choice of basis better suited for positivity questions in this setting.  Alternatively, perhaps we may still modify the conjecture to apply in this setting; for example, one possible approach would be to lower-bound the $e$-basis coefficients of certain vertex-weighted graphs of $n$ vertices and total weight $d$ by a function of the ``excess weight'' $d-n$ in such a way that this lower bound is $0$ when $d-n=0$.

In addition to $e$-positivity, a result of Gasharov provided the first major step into studying $s$-positivity of the chromatic symmetric function.  Given a poset $(P,\preceq)$, and a partition $\lm \vdash |P|$, define a \emph{P-tableau of shape} $\lm$ to be a filling of the Young diagram of shape $\lm$ with elements of $P$, each occurring exactly once, satisfying
\begin{itemize}
\item The entries are strictly increasing along rows (if $q$ is immediately to the right of $p$, then $p \preceq q$).
\item The entries are weakly increasing down columns (if $q$ is immediately below $p$, then $q \npreceq p$).
\end{itemize}
Thus, a $P$-tableau of shape $\lm$ is a generalization of the notion of a semi-standard Young tableau (with rows and columns switched).  Gasharov \cite{gash} showed:
\begin{theorem}\label{spos}
  Let $P$ be a $(3+1)$-free poset, and for each $\lm \vdash |P|$, let $f_{\lm}^P$ be the number of $P$-tableaux of shape $\lm$.  Then
$$
X_{Inc(P)} = \sum_{\lm \vdash |P|} f_{\lm}^Ps_{\lm}.
$$
In particular, claw-free incomparability graphs are $s$-positive.
\end{theorem}

Further results since have determined the value of specific coefficients in the $s$-basis expansion of $X_G$, even for graphs $G$ that are not incomparability graphs.  For example, Kaliszewski \cite{kaz} showed

\begin{theorem}\label{thm:hook}
  Let $G$ be a graph with $n$ vertices, and let $a_k(G)$ be the number of acyclic orientations of $G$ with exactly $k$ sinks.  Then when expanding $X_G$ in the basis of Schur functions, the coefficient of $s_{(m,1^{n-m})}$ is
  $$
\sum_{k=1}^n \binom{k-1}{m-1}a_k(G)
  $$
\end{theorem}

Very recently, David and Monica Wang \cite{wang} have shown a general formula for any Schur function coefficient of a chromatic symmetric function as a weighted, signed sum of special rim-hook tabloids.  Although we do not yet have results concerning the $s$-positivity of vertex-weighted graphs, we believe that the use of a deletion-contraction relation could prove useful in this context as well.

\end{subsection}

\begin{subsection}{Other Possible Applications}

There are many further possible applications of the vertex-weighted graph construction.  The authors' paper \cite{paper} extends these ideas to the bad coloring chromatic symmetric function $XB_G$ of Stanley \cite{stanley2}, and uses deletion-contraction to find many ways to construct graphs with equal chromatic symmetric function, similar to the methods of Orellana and Scott \cite{ore}.  Other possible avenues of interest may include
\begin{itemize}
\item Partition systems as described by Lenart and Ray \cite{umbra}.  Given a set $S$ and a partition $\sigma$ of $S$, they define a partition system $P$ to be a set of subsets of $S$ such that $P$ contains the empty set and all blocks of $\sigma$, and all other elements of $P$ are unions of the blocks of $\sigma$ (henceforth the blocks of $\sigma$ are called atoms of $P$, and the set of atoms is denoted $At(P)$ so we may drop mention of $\sigma$).  One example of a partition system is the \emph{independence complex} $I(G)$ of a graph $G$, where the set is $V(G)$, and $I(G) = \{A \subseteq V(G): \forall v_1,v_2 \in A, v_1v_2 \notin E(G)\}$. In particular, $I(G)$ is an abstract simplicial complex, meaning that for every $A \subseteq V(G)$, $A \in I(G) \rightarrow B \in I(G)$$ $$\forall B \subseteq A$.

For a partition system $P$ on a set $S$, one may define
$$
X_P =\sum_{\kappa} \prod_{s \in S} x_{\kappa(s)}
$$
where the sum ranges over all $\kappa: S \rightarrow \mathbb{Z}^+$ such that for each positive integer $i$, the set $\{s \in S: \kappa(s) = i\}$ is in $P$.  Then clearly $X_G = X_{I(G)}$.  Lenart and Ray showed that for any nonempty $U \in P$ that is not an atom
$$
X_P = X_{P \bk\bk U} + X_{P / U} 
$$
where $P \bk\bk U = P \bk \{W: U \subseteq W\}$ and $P / U = \{W \in P: U \subseteq W \text{ or } U \cap W = \emptyset\}$.  In the case $P = I(G)$ where $G$ is a simple graph, the atoms are all vertices of $G$, and the possible choices of $U$ are the independent sets of $G$.  Taking the particular choice of $U = v_1v_2$ where $v_1v_2$ is a nonedge of $G$ yields
$$
X_{G \bk e} = X_G + X_{P / v_1v_2}.
$$
This does not provide a direct deletion-contraction relation for $X_G$ because in $P / v_1v_2$ not all atoms are singletons, and thus the partition system does not correspond to the independence complex of a graph.  Our $X_{(G,w)}$ provides the necessary construction to generalize this result to graphs, and it may be possible to relate other results on $X_P$ to $X_{(G,w)}$.  

\item Lenart and Ray also define an \emph{umbral chromatic polynomial} $\chi^{\phi}(P;x)$ \cite{umbra, ray2} for any partition system $P$, defined as a polynomial over the algebra $\mathbb{Z}[\phi_1,\phi_2,\dots]$, where the $\phi_i$ are indeterminates.  They show that for graphs $G$, knowing $\chi^{\phi}(I(G);x)$ is equivalent to knowing $X_G$, and that $\chi^{\phi}(P;x)$ satisfies an analog of deletion-contraction (\cite{umbra}, Proposition 5.7).  As with $X_P$, this relation does not correspond to one on $X_G$ because it includes terms involving partition systems $P$ that do not correspond to graphs; however, the introduction of $X_{(G,w)}$ provides an intermediate step, and it may be of interest to determine how the vertex-weighted graph construction fits into this setting, especially in the context of the addition-contraction tree given in \cite{ray2}.  
\item The path-cycle symmetric function on digraphs defined by Chow \cite{chow}.  In particular, the most appropriate formulation to generalize is likely
  $$
\Theta_D(x_1,x_2,\dots,y_1,y_2,\dots) = \sum_{(S,\kappa)} \prod_{v_1 \text{ in a path}} x_{\kappa}(v_1) \prod_{v_2 \text{ in a cycle}} y_{\kappa(v_2)}
  $$
  where the sum ranges over path-cycle covers $S$ and colorings $\kappa$ of $D$ such that paths and cycles are monochromatic, and the paths receive distinct colors.  Perhaps there is something akin to a natural deletion-contraction relation in this context.
\item The double poset construction of Grinberg \cite{grin}.  He defines a double poset $(E, \preceq_1, \preceq_2)$ using two partial orders on the same base set, and given a weight function $w: E \rightarrow \mathbb{Z}^+$, he defines the function
$$
\Gamma(E, w) = \sum_{\pi : E \rightarrow \mathbb{Z}^+} \prod_{e \in E} x_{\pi(e)}^{w(e)}
$$
where the sum runs over all $\pi$ that are $E$-partitions, a notion that generalizes $(P, \omega)$-partitions.  Results on $\Gamma(E,w)$ are in some cases directly applicable to $X_{(G,w)}$.  For example, as an alternate proof of \eqref{eq:inv}, one could fix an orientation $\gamma$ and pass to quasisymmetric functions.  Then, upon swapping out the symmetric function involution for the related antipode on quasisymmetric functions, \eqref{eq:inv} reduces to a special case of (\cite{grin}, Theorem 4.2).  There are likely other relations between these functions waiting to be discovered.
\end{itemize}

\end{subsection}

\begin{subsection}{A Brief Note On Weighted Trees}

As a final note we would be remiss not to mention one of the other major open problems of the chromatic symmetric function, the tree isomorphism conjecture, inspired by a question of Stanley \cite{stanley}:

\begin{conj}\label{conj:tree}

If $G$ and $H$ are trees, and $X_G = X_H$, then $G$ and $H$ are isomorphic.

\end{conj}

This conjecture has been shown to be true for trees with up to $29$ vertices $\cite{heil}$.  A natural question is whether it is possible that a stronger statement holds, that the chromatic symmetric function distinguishes vertex-weighted trees.  This is false, as can be seen in the following example from $\cite{loebl}$ by comparing

(a) The five-vertex path with vertex weights $1,2,1,3,2$ in that order, and

(b) The five-vertex path with vertex weights $1,3,2,1,2$ in that order.

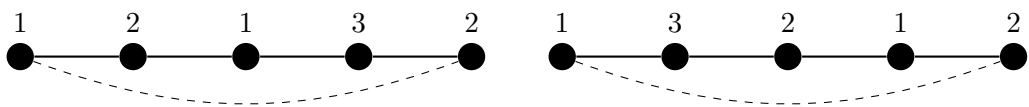
\begin{figure}[hbt]
\begin{center}
\begin{tikzpicture}[scale=1.5]

\node[label=above:{$1$}, fill=black, circle] at (0, 0)(v1){};
\node[label=above:{$2$}, fill=black, circle] at (1, 0)(v2){};
\node[label=above:{$1$}, fill=black, circle] at (2, 0)(v3){};
\node[label=above:{$3$}, fill=black, circle] at (3, 0)(v4){};
\node[label=above:{$2$}, fill=black, circle] at (4, 0)(v5){};

\draw[black, thick] (v1) -- (v2);
\draw[black, thick] (v2) -- (v3);
\draw[black, thick] (v3) -- (v4);
\draw[black, thick] (v4) -- (v5);
\draw[black, dashed] (v1) edge [bend right = 20] (v5);

\end{tikzpicture}
\hspace{0.5cm}
\begin{tikzpicture}[scale=1.5]

\node[label=above:{$1$}, fill=black, circle] at (0, 0)(v1){};
\node[label=above:{$3$}, fill=black, circle] at (1, 0)(v2){};
\node[label=above:{$2$}, fill=black, circle] at (2, 0)(v3){};
\node[label=above:{$1$}, fill=black, circle] at (3, 0)(v4){};
\node[label=above:{$2$}, fill=black, circle] at (4, 0)(v5){};

\draw[black, thick] (v1) -- (v2);
\draw[black, thick] (v2) -- (v3);
\draw[black, thick] (v3) -- (v4);
\draw[black, thick] (v4) -- (v5);
\draw[black, dashed] (v1) edge [bend right = 20] (v5);

\end{tikzpicture}

\end{center}
\caption{Weighted trees with the same chromatic symmetric function}
\label{fig:123}

\vspace{0.5cm}

\end{figure}

It is seen easily that these are not isomorphic as vertex-weighted graphs.  To see that nonetheless they have the same chromatic symmetric function, we apply the addition form \eqref{eq:edgeadd} of the deletion-contraction rule to the non-edge represented by the dashed line.  Then the chromatic symmetric function of both (a) and (b) is the same as that of a five-vertex cycle with vertex weights $1,2,1,3,2$ cyclically, added to that of a four-vertex cycle with vertex weights $3,3,2,1$ cyclically.

However, in this example the two underlying unweighted trees are isomorphic.  We do not know of an example of two vertex-weighted trees $(T,w)$ and $(T',w')$ with $T$ and $T'$ nonisomorphic as unweighted trees, but with $X_{(T,w)} = X_{(T',w')}$.

\end{subsection}

\end{section}

\begin{section}{Acknowledgments}

  The authors would like to thank Greta Panova for introducing us to the chromatic symmetric function and for numerous helpful discussions, Stephanie van Willigenburg for directing us to \cite{wolfgang}, and Jos\'{e} Aliste-Prieto, Darij Grinberg, and Bruce Sagan for helpful comments.  We are thankful to the anonymous referee for many helpful suggestions.

  The authors would also like to acknowledge Brendan McKay's webpage of combinatorial data \cite{mckay}.  We have made extensive use of his databases of graphs and trees to develop and test conjectures, both in this work and in \cite{paper}.

   The second author is supported by the National Science Foundation under Award No. DMS-1802201.

\end{section}

\end{document}